\documentclass[12pt]{article}
\usepackage{amsmath,epsfig,subfigure,amsthm,amsfonts,amsbsy,amssymb,latexsym,amsxtra}
\usepackage[usenames]{color}
\usepackage{hyperref}

\usepackage{pst-infixplot}
\usepackage{pst-plot}
\usepackage{pstricks}

\usepackage[sort]{natbib}

\makeatletter \@addtoreset{equation}{section}

\newcommand{\beq}{\begin{equation}} 
\newcommand{\eeq}{\end{equation}}
\newcommand{\bed}{\begin{displaymath}}
\newcommand{\eed}{\end{displaymath}}
\newcommand{\ben}{\begin{eqnarray*}}
\newcommand{\een}{\end{eqnarray*}}

\def\F{{\cal F}}
\def\op{{\cal L}}
\def\cd{(\cdot)}
\def\M{{\cal M}}
\def\rr{{\mathbb R}}

\newcommand{\sxal}{\ensuremath{s,x,\al}}
\newcommand{\sxalz}{\ensuremath{s_0,x_0,\al_0}}

\newcommand{\dist}{\text{dist}}
\newcommand{\e}{\varepsilon}
\newcommand{\vphi}{{\varphi}}

\newcommand{\sz}{\ensuremath{s_0}}
\newcommand{\xz}{\ensuremath{x_0}}
\newcommand{\al}{\alpha}

\newcommand{\pr}{{\mathbf P}}
\newcommand{\ex}{{\mathbf E}}
\newcommand{\wdt}{\widetilde}

\newtheorem{thm}{Theorem}[section]
\newtheorem{prop}[thm]{Proposition}
\newtheorem{lem}[thm]{Lemma}

\theoremstyle{definition}
\newtheorem{rem}[thm]{Remark}
\newtheorem{defn}[thm]{Definition}
\newtheorem{exm}[thm]{Example}
\newtheorem{assumption}{Assumption A\!\!}

\newcommand{\thmref}[1]{Theorem~{\rm \ref{#1}}}
\newcommand{\lemref}[1]{Lemma~{\rm \ref{#1}}}

\newcommand{\propref}[1]{Proposition~{\rm \ref{#1}}}
\newcommand{\assumpref}[1]{Assumption A{\rm \ref{#1}}}

\newcommand{\exmref}[1]{Example~{\rm \ref{#1}}}

\newcommand{\set}[1]{\left\{#1\right\}}
\newcommand{\abs}[1]{\left\vert #1\right\vert}

\newcommand{\disp}{\displaystyle}

\newcommand{\bedd}{\bed\begin{array}{l}}
\newcommand{\eedd}{\end{array}\eed}

\def\({\left(}
\def\){\right)}

\def\one{{\hbox{1{\kern -0.35em}1}}}

\newcommand{\bea}{\bed\begin{array}{rl}}
\newcommand{\eea}{\end{array}\eed}
\newcommand{\ad}{&\!\!\!\disp}
\newcommand{\aad}{&\disp}
\newcommand{\barray}{\begin{array}{ll}}
\newcommand{\earray}{\end{array}}

\title{Optimal Control of Risk Process in a Regime-Switching Environment}

\author{Chao Zhu\thanks{Department of Mathematical
Sciences, University of Wisconsin-Milwaukee, Milwaukee, WI 53201,
{\tt zhu@uwm.edu}.}}

\begin{document}
\maketitle

\begin{abstract}
This paper is concerned with cost optimization of an insurance company. The surplus of the insurance company is modeled by a controlled regime-switching diffusion,
where the regime-switching mechanism provides the fluctuations of the random environment.
  The goal is to find an optimal control that minimizes the total cost up to a stochastic exit time.
A weaker sufficient  condition than that of \cite[Section V.2]{FlemingS} for the continuity of the value function is obtained. Further, the value function is shown to be a viscosity solution of a Hamilton-Jacobi-Bellman equation.


\end{abstract}

\setlength{\baselineskip}{0.232in}

\section{Introduction}\label{sect-introduction}
Recently 
the optimal risk control and dividend distribution problems have   drawn growing attention
from 
  researchers.
  Some recent developments 
   can be found in
\cite{Cadenillas-06,Choulli,Irgens-05,Paulsen-97,Paulsen-05,Schmidli-01,Schmidli-02,Taksar-03,Taksar-07,Touzi-00} and the references therein.
In those works, the liquid assets   of the insurance companies are modeled by
some stochastic processes (usually linear diffusions or jump diffusions).
 At any time $t$, the insurance companies can choose different
 business activities such as reinsurance, investment, dividend payment, etc.
 The decisions are based upon the information available to them by time $t$ as well as  the pre-given  economic or political criteria.
  Different business activities lead to different dynamics in the evolution of the surplus and hence
   different economic or political returns.
This sets a scene for an optimal stochastic control model for 
the surplus of the insurance company.
In the literature,
the most commonly used  criteria 
are:
(i) maximizing expected utility at or up to the time of ruin
 (\cite{Irgens-05,Touzi-00}),
(ii) minimizing the probability of ultimate ruin
(\cite{Schmidli-01,Schmidli-02,Taksar-03}),
 and (iii)
maximizing the cumulative expected discounted dividends (\cite{Cadenillas-06,Choulli,Paulsen-97}).

In contrast to the aforementioned references, in this work, we
 aim to investigate the problem of cost optimization  for an insurance company in a regime-switching environment using
  stochastic analysis and stochastic control theories.
  As we shall see shortly, this is a nonlinear optimal control problems in the setting of regime-switching diffusion.
In addition to the usual operating cost of an insurance company such as corporate debt, bond liability, loan amortization,   etc,
  in practice, almost every insurance claim is accompanied by a certain amount of business cost
resulting from claim appraisal, investigation, settlement negotiation, and so on.
Minimizing the cost may increase the profit of the insurance company and lower premiums for its customers.

We should also note that the word cost can be used in a general sense: it may represent any monetary amount such as claims, penalties, dividends, utilities, and so on.
As illustrated in \cite{Cai-Feng}, 
 the total  discounted  cost actually   covers
 a number of ruin-related quantities frequently analyzed in ruin literature such as the expected present value of penalty at ruin and the total dividends under various dividend strategies.
Therefore this work can be applied to a broad range of aspects of risk management
such as utility and cumulative dividends maximization.

Another feature of this work is the consideration of regime-switching. Most of the existing literature on optimal control of risk processes are based on the framework of diffusion approximation model (\cite{Grandell-91}).
That is, the surplus of an insurance company is usually modeled or approximated by a (jump) linear diffusion process.
Roughly speaking, if the surplus of the insurance company is much larger than the individual claims, then the classical homogeneous Poisson model can be approximated by a diffusion model.
Along another line,
\cite{Asmussen-89} proposed a Markovian-modulated risk model. The model  is a hybrid system, in which continuous dynamics   are intertwined with discrete events.
More specifically,
 the evolutions of the surplus (continuous dynamics)  is subject to jumps or switches of the economic or political environment (discrete events), and the dynamics of the surplus in different environments are markedly different.
As demonstrated in \cite{Asmussen-89} (see also \cite{Yang-Yin04}), this model  can capture the features of insurance policies that depend on the economic or political environment changes.
  The states of the discrete event process can model   for example, certain types of epidemics in health insurance,   weather types in automobile insurance,  the El Ni\~{n}o/La Ni\~na phenomenon in property insurance, or economic conditions in unemployment insurance.
 Also, in many practical situations, people are more concerned with the short term results of the business activities. For instance, the manager and/or shareholders of an insurance company want to determine the short-term benefits of a particular business activity.
Inspired by these arguments, we consider a controlled surplus
process modeled by a   regime-switching diffusion (also called a hybrid diffusion in the literature)
 over a finite time horizon $[s,T]$, where $T>0$ is a  fixed constant and $s\in [0,T)$.


By the nature of the risk process, the insurance company  may default at some finite time; at which point we say that the surplus is ruined and denote the ruin time by $\tilde \tau$.
Consequently, we need only to consider the total cost up to the ruin time or the terminal time $T$, whichever comes first.
 That is, our control problem is over the interval $[s, T\wedge \tilde \tau]$, with $T\wedge \tilde \tau= \min \set{T, \tilde \tau}$ being a random time (stopping time).
 This is generally termed as   exit time control or stochastic control with exit time in the literature (\cite{FlemingS}).
 As we shall see in Example \ref{exm-discontinuous-value-function},
the stopping time $T\wedge \tilde \tau$   depends on the initial surplus $X(s) =x$. As a result, the value function defined in \eqref{eq:value-sw}
is not necessarily continuous with respect to $x$.
See Example \ref{exm-discontinuous-value-function} for detailed discussions.
Then, a problem of great interest is: Under what condition(s), is the value function continuous?
\cite[p. 202]{FlemingS} proposed a condition on the drift of the underlying (1-dimensional) controlled diffusion under which the continuity of the value function is guaranteed.
The condition in \cite{FlemingS} was recently generalized in \cite{Bayraktar} by considering both the drift and the diffusion coefficients.
In this work, under the more general setting of controlled regime-switching diffusion, we propose a new condition in terms of the regularity of the boundary point to obtain the continuity of the value function.
Our condition is another generalization of the one in \cite{FlemingS}.
See \thmref{thm-continuity}, \propref{prop-about-A3}, and Example \ref{exm2} for more details.

We emphasize that the continuity of the value function is very important
and useful for the following reasons. Firstly, with the continuity of the value function, one has the dynamic programming principle
(\cite{FlemingS}), which, in turn, helps to establish the viscosity solution property for the value function.
Secondly, the continuity of the value function plays a vital role in the study of
numerical approximation to the value function.
To illustrate, let $X$ be a controlled continuous-time stochastic process and $V$ denotes the associated value function.
 Typically, one constructs a  controlled locally consistent discrete sequence $\{X_n^h, n\in \mathbb N\}$ and find the associated value function $V^h$ for the
discrete problem, where $h>0$ is the stepsize of the discretization.
If the value function $V$ is not continuous, then as $h\downarrow 0$, $V^h$ may not converge to   $V$,
even though  $X^h$ converges to $X$ in some suitable sense,
  where $X^h$ is  the   continuous parameter interpolated process of  $\set{X_n^h, n\in \mathbb N}$. In other words, $X^h$ approximates $X$, but $V^h$ may still differ  significantly from $V$.
  See Example \ref{exm-discontinuous-value-function} and also \cite{Kushner-D} for more details.
  In fact, this work is largely motivated by this aspect of consideration.


The classical approach to stochastic control problem is the verification theorem. 
Typically, this approach requires   the value function to be  sufficiently smooth.
In such a case, it can be shown that  the value function
is a {\em classical} solution of a nonlinear partial differential equation of Hamilton-Jacobi-Bellman type.
See \cite{YongZ} for details.
However, the smoothness assumption is rather restrictive.
As we mentioned earlier, there are many examples where the value function is not necessarily sufficiently smooth.
In fact, due to the dependence of the terminal time $T\wedge \tilde\tau$ on the initial data $X(s)=x$,
 the value function in our setup may  not be even continuous.
 Then how can we characterize the value function?
In this paper, we use the notion of   viscosity solution.
With the aid of the dynamic programming principle, we prove that  the value function is a {\em viscosity} solution of the Hamilton-Jacobi-Bellman equation.
Since the seminal work \cite{Crandall-Lions-83,CIL92,Lio83b} and others, the viscosity solution characterization have been exploited   in many  control problems under various settings.
But the related results in the content of regime-switching diffusion is relatively scarce. Moreover,
 the proof for our case is not a trivial extension of the existing results. The presence of regime-switching adds much difficulty in the proof.
 See \thmref{thm-viscosity-soln} for more details.

At this point, it is   worth mentioning that thanks  to the  capability of  delineating  the inherent
randomness of many real-world applications, regime-switching diffusions
   have been used in
   a wide range of areas such as finance, biology, insurance, etc.
 See for instance \cite{Mariton,Hespanha05,MaoY,YZ-10}  and references therein for many prototypical  examples of applications of such stochastic processes.
  On the control aspect,
 optimal control problems of regime-switching diffusions have received growing attention lately.
 To name just a few, \cite{Li-03,Rami-01,Taksar-Z,YinLZ,Zhou-Yin,Song-S-Z} among others.


The rest of the paper is arranged as follows. We formulate the problem in Section \ref{sect-formulation}.
  Section \ref{sect-continuity} is devoted to the continuity of the value function. We obtain several sufficient conditions for continuity.
In Section \ref{sect-viscosity}, we show that the value function is a viscosity solution of the HJB equation \eqref{eq:HJB}.
We conclude the paper with a few remarks in Section \ref{sect-conclusions}.
Appendix \ref{sect-appendix} provides a result on regularity of the boundary point.

A few words about notations are necessary at this point. We use $I_A$ to denote the indicator function of a set $A$.
If $a, b \in \rr $, then $a\wedge b :=\min\set{a,b}$.   Throughout the paper, $K$ is a generic positive constant whose exact value may differ in different appearances. For any $x_0 \in \rr$ and $r>0$, $B(x_0,r)= \set{x\in \rr: \abs{x-\xz} < r}$.

\section{Problem Formulation}\label{sect-formulation}
Let $X(t)$ denote the surplus of a large insurance company at time $t \in
[s,T]$, where $T> 0$ is fixed and $s \in [0,T)$.
As we indicated in Section \ref{sect-introduction}, the surplus process $X$ often displays abrupt changes according to
 different economic,  political, or natural environments facing the insurance
company.   Following \cite{Asmussen-89}, we use a continuous time Markov chain
$\al\cd$ to model the variations of the external environments of the insurance company.
For simplicity, we assume the Markov chain has a finite state space $\M=\set{1, \dots,
m}$ and is  generated by $Q=(q_{ij})$,
 that is,
\beq\label{Q-gen} \pr\set{\al(t+ \Delta t)=j|
\al(t)=i,\al(s),s\le t}=\begin{cases}q_{ij}
\Delta t + o(\Delta t),\ &\hbox{ if }j\not= i,\\
1+ q_{ii}\Delta t + o(\Delta t), \ &\hbox{ if }j=i,
\end{cases}\eeq where $q_{ij}\ge 0$ for $i,j=1,\dots,m$
with $j\not= i$ and $\sum^m_{j=1}q_{ij}=0$ for each $i=1,\dots,m$.
At any time $t$, we denote by  $u(t)$  one of  the possible  business activities
available for the insurance company. For instance, $u(t)$ may represent
a reinsurance strategy, an investment plan, or a dividend payment scheme, etc.

Suppose $X$ satisfies the following stochastic differential equation with regime
switching:
\begin{equation}\label{eq:risk-sw} dX(t) = b(t,X(t),\al(t),u(t)) dt + \sigma(t,
X(t),\al(t), u(t)) dw(t),\end{equation}
with initial conditions
\begin{equation}\label{eq:initial}
X(s) =x>0 \ \text{ and } \ \al(s) =\al\in \M, \end{equation}
where $w\cd$ is a standard Brownian motion independent of the Markov chain $\al\cd$.
Note that the independence between $w$ and $\al$ is a standard assumption in the
literature (\cite{MaoY}). Denote
$ {\cal F}_t :=\sigma\set{w(r),\al(r): 0\le r \le t}.$
Without loss of generality, we assume that $\F_0$ contains all
$\pr$-null sets.
Suppose throughout the paper that the control policy $u$ is
  $\{\F_t\}$-adapted and that for any $t$,
  $   u(t)  \in {U} $, where  $U$ is a compact subset of $\rr$.
Any control $u$ satisfying the above conditions is
  said to be an {\em admissible control}. Let $\cal U$ denote the collection of all
  admissible controls.

 Let $\tilde \tau:=\inf\set{t\ge s: X(t)=0}$ denote  the ruin time
  and set $\tau:=T \wedge \tilde \tau $.
For a given control $u \in \cal U$, the expected total  cost is
\begin{equation}\label{eq:cost-sw}
J(s,x,\al,u\cd)= \ex_{s,x,\al}\left[ \int_s^\tau
l(t,X(t),\al(t),u(t))dt +  g(\tau, X(\tau),\al(\tau))\right],
\end{equation}
where $l: [0,T]\times \rr\times \M \times { U} \mapsto
\rr $ represents the
  running cost,
 $g: [0,T] \times  [0,\infty) \times \M \mapsto \rr $ is the terminal cost,
 and $\ex_{s,x,\al}$ denotes the expectation with respect to the probability law such that the regime-switching diffusion $X(t)$ in \eqref{eq:risk-sw} starts with initial condition  specified in \eqref{eq:initial}.
 As we mentioned in Section \ref{sect-introduction}, we use the word cost in the
 general sense throughout the paper.
 Hence we allow the functions $l$ and $g$ to be negative.

 The goal is to find an optimal control $u^* \in \cal U$ that minimizes the
total   cost
\begin{equation}\label{eq:value-sw}
V(s,x,\al)=J(s, x,\al,u^*)= \inf_{u\in{\cal U}} J(s,x,\al,u), \ \ \forall
(s,x,\al) \in [0,T) \times (0,\infty)\times \M.
\end{equation}
Note that the terminal and boundary conditions are
\beq\label{eq-value-terminal-condition}
V(T,x,\al)= g(T,x,\al), \ \  \forall  (x,\al) \in (0, \infty)\times \M,
\eeq
and \beq\label{eq-value-boundary-condition}
V(s,0,\al) = g(s,0,\al),\ \  \forall  (s,\al) \in [0,T]\times \M.
\eeq

Throughout the paper, we assume
\begin{assumption}\label{assump-ito-condition}  For each $\al \in \M$,
the functions $b(\cdot, \cdot, \al, \cdot)$, $\sigma(\cdot, \cdot, \al, \cdot)$, $l(\cdot, \cdot, \al, \cdot)$, and $g(\cdot, \al)$ are uniformly continuous. Moreover, there exist positive constants $\kappa_0 $ and $p$ such that  for any $t\in [0,T]$, $x,y \in \rr$, $\al \in \M$, and $u \in U$, we have
 \begin{equation}\label{ito-condition}\barray \ad \abs{\vphi(t,x,\al,u)-\vphi(t,y,\al,u)}   \le \kappa_0 \abs{x-y}, \text{ for } \vphi=b,\sigma, l, \text{ and } g, \\ \ad \abs{b(t,x,\al,u)} + \abs{\sigma(t,x,\al,u)} \le \kappa_0(1+ \abs{x}), \\ \ad \abs{l(t,x,\al,u)} + \abs{g(t,x,\al)} \le \kappa_0 (1+ \abs{x}^p).
  \earray \end{equation}
\end{assumption}
It is well known (\cite{MaoY}) that under \assumpref{assump-ito-condition}, for any $u\in\cal U$ and any $(s,x,\al) \in [0,T)\times (0,\infty) \times \M$, there exists a unique solution $X$ to \eqref{eq:risk-sw} with initial condition \eqref{eq:initial} under the control $u$. Moreover, $J$ in \eqref{eq:cost-sw}
is well-defined.
In the sequel, we denote such a solution by $X^{s,x,\al}$ or $X^{s,x,\al;u}$ if the emphasis on the initial conditions and the control is needed.
 Similarly, $\al^{s,\al}$ denotes the Markov chain   with initial condition $\al(s)=\al$.


For convenience of later presentations, we introduce the operator $\op^u_t$.
For any $h(t,\cdot, \al) \in C^2, t \in [0,T], \al \in
\M$ and $u \in U$, we define
\begin{equation}\label{eq-operator-defn}
\op^u_t h(t,x,\al) = b(t,x,\al,u) \frac{\partial }{\partial x}h(t,x,\al) +
\frac{1}{2} \sigma^2(t,x,\al,u) \frac{\partial^2 }{\partial x^2}h(t,x,\al) +
\sum_{j=1}^m  q_{\al j}h(t,x,j).
\end{equation}

The
following verification theorem can be established using the standard argument as in \cite{FlemingS} and \cite{YongZ}, together with the generalized It\^o formula (\cite{MaoY}). We shall omit the  proof  for brevity.
\begin{thm}\label{thm-verification}
Suppose there exists a  function $\varphi: [s,T]\times \rr_+ \times \M \mapsto
\rr_+$ such that
\begin{itemize}\item[{\em (i)}] $\varphi(\cdot, \cdot, \al) \in C^{1,2} $ for each $\al \in \M$,
\item[{\em (ii)}] $\varphi$  satisfies the polynomial growth condition, that
is, for some positive constants $p$ and  $K$, we have
    \begin{equation}\label{eq:poly-growth}
    \abs{\varphi(t,x,\al)} \le K(1+ |x|^p), \text{ for any } t\in [s,T] \text{
    and } \al \in \M, \end{equation}
\item[{\em (iii)}] $\varphi$  satisfies the Hamilton-Jacobi-Bellman {\em
(HJB)} equation
\begin{equation}\label{eq:HJB} \begin{cases} & \hspace*{-10pt}
\dfrac{\partial }{\partial t}\varphi (t,x,\al) + \disp\min_{u \in U} [\op^u_t
\varphi (t,x,\al) + l(t,x,\al, u)]=0, \\ & \hfill t\in (s,T), x >0, \al \in \M,
\\
&
 \hspace*{-10pt}\varphi(s,0,\al) =g(s,0,\al), \   \varphi(T,x,\al)= g(T,x,\al), \
 s\in [0,T), x> 0, \al \in \M. \end{cases}
\end{equation}
\end{itemize}
Then $\vphi(s,x,\al) \le J(s,x,\al, u\cd)$ for any initial condition $(s,x,\al)
\in [0,T) \times (0,\infty) \times \M$ and any admissible feedback control
$u\cd$.

Moreover, if $u^*\cd$ is an admissible feedback control such that
\begin{equation}\label{eq:optimal-HJB}\begin{aligned}
& \frac{\partial}{\partial t} \vphi(t,x,\al) +
\op^{u^*}_t  \varphi(t,x,\al) + l(t,x,\al, u^*(t,x,\al))   \\ & \ \   = \frac{\partial}{\partial t} \vphi(t,x,\al) + \min_{u \in U}
[\op^u_t  \varphi (t,x,\al) + l(t,x,\al,u)]
\\ & \ \ =0,
\ \ \qquad   \forall (t,x,\al) \in (s,T) \times (0,\infty) \times \M,\end{aligned}
\end{equation} Then $\varphi(s,x,\al) =V(s,x,\al)$ for all $(s,x,\al) \in [0,T)
\times (0,\infty) \times \M$ and $u^*\cd$ is an optimal control.
\end{thm}

\section{Continuity}\label{sect-continuity}
\thmref{thm-verification} provides conditions under which a sufficiently smooth function coincides with the value function.
 In particular, it indicates that if $\vphi$ solves the HJB equation \eqref{eq:HJB}, then it provides a lower bound for  the value function $V$. In addition, if
 we can find a feedback control $u^*\cd$ satisfying \eqref{eq:optimal-HJB},
 then $\vphi=V$ and $u^*\cd$ is an optimal control.
 It is natural to ask   whether the converse is true: ``Does the value function $V$ defined  in \eqref{eq:value-sw} always satisfy \eqref{eq:HJB}?''
 In general, the answer is no, since the value function $V$ is not necessarily smooth enough (with respect to the variables $s$ and $x$).
 More specifically, in our setup, both the stopping time $\tau$ and the control $u$ may depend on the initial value $X(s)=x$.
 Consequently, the value function may not be even continuous.
  To illustrate, we consider the following uncontrolled deterministic system, where the value function is discontinuous.
 The example is inspired by the tangency problem presented in \cite[pp. 277--279]{Kushner-D}.

\begin{exm}\label{exm-discontinuous-value-function}
Consider \beq\label{exm1-state} \begin{cases}  dX(t)=2(t-1)dt, \quad t\in [s,2] \\  X(s)=x>0, \end{cases}\eeq where $s\in[0,2]$.
The solution of \eqref{exm1-state} is
\bed X(t) =X^{s,x}(t)= (t-1)^2 + x -(s-1)^2, \ \ t\in[s,2).\eed
Let $\tau=\inf\set{t > s: X(t)=0}\wedge 2$ and   $V(s,x)= \tau$.
Note that \assumpref{assump-ito-condition} is satisfied for this example.
Consider the case when $s\in [0,1]$. Then it is obvious that $X(t)$ first decreases to its minimum $x-(s-1)^2$ then increases to $\infty $
as $t\to \infty$. As a result, we have
\bed \tau \begin{cases} = 2, & \text{ if }x-(s-1)^2>0, \\ \le 1, & \text{ if }x-(s-1)^2\le 0.\end{cases}\eed
Hence it follows that the value function $V(s,x)= \tau$ is not continuous on the parabola $\set{(s,x)\in [0,1]\times (0,\infty): x-(s-1)^2=0}.$
See the demonstration in Figure \ref{fig1}.
\begin{center}
\begin{figure}[htp]
\psset{unit=2.8cm}
\begin{pspicture}(-1.9,-.15)(3,1.5)
\psline[linewidth=.6pt,arrows=->] (-.35,0)(2.5,0)
\psline[linewidth=.6pt,arrows=->] (0,-.25)(0,1.45)
\psline[linewidth=0.5pt,linestyle=dotted](2,1.3)(2,0)

\psPlot[linewidth=1.5pt,linecolor=blue]{0.25}{1}{(x-1)*(x-1)}
\uput[d](0.20,.3){{$X$}}
\psline[linewidth=.3pt,arrows=->] (0.20,.3)(0.25,.5)

\psPlot[linewidth=1.5pt,linecolor=red,linestyle=dotted]{0.25}{2}{(x-1)*(x-1)+0.05}
\uput[d](0.45,.95){{$\tilde X$}}
\psline[linewidth=.3pt,arrows=->] (0.4,.73)(0.3,.605)

\uput[d](1,0){\large{$\tau =1$}}
\uput[d](2,0){\large{$\tilde\tau=2$}}
\uput[r](2.5,0){\large{$t$}}
\end{pspicture}
\caption{Discontinuous Value Function}
\end{figure}
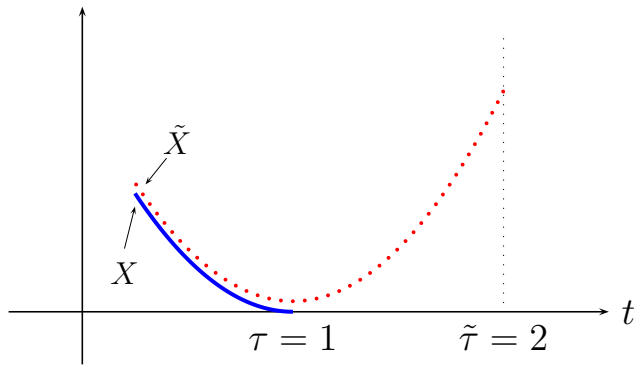\label{fig1}
\end{center}
\end{exm}

Example \ref{exm-discontinuous-value-function} naturally motivates us to the following question: ``Under what conditions is
 the value function continuous?''
To answer this question,  we follow the treatment in \cite[Chapter V]{FlemingS}.
 We first consider an auxiliary control problem, whose  value function $V^\e$ is continuous. Then we propose conditions
  under which $V^\e$ converges uniformly to the original value function $V$ as $\e\downarrow 0$, from which the  continuity of $V$ is established.
 As we shall see in \thmref{thm-continuity}, \propref{prop-about-A3}, and \exmref{exm2}, our condition is more general than that in \cite[Section V.2]{FlemingS}.

Let $\psi: \rr \times \M \mapsto \rr  $ be a  function such that
\beq\label{psi-Lip}  \abs{\psi^+(x,\al)-\psi^+(y,\al)} \le L \abs{x-y}, \text{ for
any  } x,y \in \rr, \al \in \M,\eeq where   $L> 0$  is a constant.
For any $\e>0$, we define
\beq\label{eq-gamma-defn}  \Gamma^\e (t) := \exp\set{- \frac{1}{\e} \int_s^t
\psi^+(X(r),\al(r)) dr}.  \eeq
Note that $\Gamma$ depends on the processes $X$, $\al$, and the underlying control $u$ as well. But for notational
simplicity, we have omitted those dependence.
Consider the auxiliary control problem
\begin{align}\label{eq-cost-aux}
J^\e(s,x,\al, u\cd) & = \ex_{s,x,\al} \left[ \int_s^T  \Gamma^\e(t)
l(t,X(t),\al(t),u(t))dt +  \Gamma^\e(T) g(T,X(T),\al(T))\right],
\\ \label{eq-V-aux}
 V^\e(s,x,\al)& = \inf_{u\cd  \in \cal U} J^\e (s,x,\al, u\cd).
\end{align}

\begin{lem}\label{lem-gamma1-2} Assume Assumption {\em A\ref{assump-ito-condition}} and \eqref{psi-Lip}.
 For any $u  \in \cal U$, denote  $X_i(t) = X^{s,x_i,\al; u}(t)$ and   
 $$\Gamma^\e_i(t) = \exp\set{-\frac{1}{\e}\int_s^t \psi^+(X_i(r),\al(r))dr},\ \ t \in [s,T],$$ where
 $s\in [0,T)$,
$x_i  > 0 $,  $\al \in \M$, and $i=1,2$.
Then we have
  \beq\label{eq-L2-estimate-x12}
   \ex  \sup_{t\in[0,T]} \abs{X_1(t)-X_2(t)}^2 \le K\abs{x_1-x_2}^2,\eeq
and \beq\label{eq-Gamma-cont}
\abs{\Gamma^\e_1 (t)- \Gamma^\e_2 (t)} \le  \frac{L}{\e} (t-s) \sup_{r
\in[s,t]}
\abs{X_1(r)-X_2(r)}.
\eeq
\end{lem}
\begin{proof} Note that   virtually the same argument as that of \cite[Lemma 2.14]{YZ-10} yields \eqref{eq-L2-estimate-x12}. Therefore it remains to prove \eqref{eq-Gamma-cont}.
It is easy to see that  $\abs{e^{-a} - e^{-b}} \le \abs{a-b}$ for any $a,b
\ge 0$. Thus it follows from \eqref{psi-Lip} that
\bed \barray
  \abs{\Gamma^\e_1 (t)- \Gamma^\e_2 (t)} \ad =
  \abs{\exp\set{-\frac{1}{\e}\int_s^t \psi^+(X_1(r),\al(r))dr} -
  \exp\set{-\frac{1}{\e}\int_s^t \psi^+(X_2(r),\al(r))dr}} \\[1.5ex]
\ad \le \frac{1}{\e} \abs{\int_s^t  \left[\psi^+(X_1(r),\al(r)) -
\psi^+(X_2(r),\al(r))\right] dr} \\[1.5ex]
\ad \le   \frac{1}{\e}  \int_s^t L  \abs{X_1(r) -X_2(r) }dr \\[1.5ex]
\ad \le  \frac{L}{\e} (t-s) \sup_{r \in[s,t]}
\abs{X_1(r)-X_2(r)}.
\earray \eed This completes the  proof.
\end{proof}

\begin{thm}\label{thm-cont-aux} Under Assumption  {\em A\ref{assump-ito-condition}},
  the function  $V^\e (s,x,\al)$ is continuous  with respect to the variables $s$ and $x$ for each $\al \in \M$,.
 \end{thm}

\begin{proof} We divide the proof into several steps.

Step 1.
For $\phi= l,g$ and $t \in [0,T]$, with notations as in \lemref{lem-gamma1-2}, it follows from the Cauchy-Schwartz inequality and  \assumpref{assump-ito-condition} that
\bea
\ad \ex  \abs{\Gamma^\e_1(t) \phi(t, X_1(t), \al(t), u(t) )-
\Gamma^\e_2(t) \phi(s, X_2(s), \al(s), u(t) )  } \\[1ex]
\aad  \le \ex \abs{(\Gamma_1^\e (t) - \Gamma_2^\e(t)) \phi(t, X_1(t), \al(t),
u(t) )} \\[1ex]
 \aad \quad + \ex \abs{\Gamma^\e_2(t)[ \phi(t, X_1(t), \al(t), u(t)
)- \phi(t, X_2(t), \al(t), u(t) )]} \\[1ex]
\aad \le  \ex^{1/2} \abs{\Gamma_1^\e (t) - \Gamma_2^\e(t)}^2 \ex^{1/2}\abs{
\phi(t, X_1(t), \al(t), u(t) )}^2 \\[1ex]
\aad \quad +  \ex^{1/2} \abs{\Gamma^\e_2(t)}^2 \ex^{1/2} \abs{\phi(t, X_1(t),
\al(t), u(t) )- \phi(t, X_2(t), \al(t), u(t) )}^2  \\[1ex]
\aad \le K \ex^{1/2} \abs{\Gamma_1^\e (t) - \Gamma_2^\e(t)}^2 \ex^{1/2} (1+ \abs{X_1(t)}^{2p}) + K \ex^{1/2}\abs{X_1(t)-X_2(t)}^2.
\eea
Note that by virtue of \cite[Theorem 3.24]{MaoY}, we have \beq\label{eq-Lp-estimate}\ex  \sup_{t\in[0,T]} \abs{X_1(t)}^{2p}  \le K= K(x_1, T, p).\eeq
This, together with \lemref{lem-gamma1-2}, leads to
\beq\label{eq:ineq-gamma-phi}
\barray
\ad  \ex  \abs{\Gamma^\e_1(t) \phi(t, X_1(t), \al(t), u(t) )-
\Gamma^\e_2(t) \phi(s, X_2(s), \al(s), u(t) )  } \\
\aad \le K \frac{L}{\e} (t-s) \ex^{1/2} \( \sup_{r\in
[s,t]}\abs{X_1(r)-X_2(r)}\)^2
+ K \ex^{1/2} \abs{X_1(t)-X_2(t)}^2
\\ \aad \le K\abs{x_1 -x_2},\earray
\eeq
where $K=K(\e,x_1, T, L, p)$ does not depend on 
$x_2 $, $t$, or $u$.

Step 2. Now it follows from \eqref{eq:ineq-gamma-phi} that
\beq\label{eq-V-x1-x2}\barray
\ad \abs{V^\e(s,x_1, \al) -V^\e(s,x_2, \al)  }\\  \aad \le \sup_{u \in \cal U}
\ex \bigg[\int_s^T  \abs{ \Gamma^\e_1 (t) l(t,X_1(t),\al(t), u(t)) dt -
\Gamma^\e_2 (t) l(t,X_2(t),\al(t), u(t))} dt
\\ \aad \quad \qquad \quad \hfill +   \left|\Gamma^\e_1 (t) g(T, X_1(T)) -
\Gamma^\e_2 (t) g(T, X_2(T)) \right|\bigg]
 \\
\aad \le K\abs{x_1 -x_2}.\earray
\eeq This shows that $V^\e(s, \cdot, \al)$ is continuous for any $s \in [0,T]$ and $\al\in\M$.

Step 3. Next we prove that $V^\e$ is continuous with respect to $s$
as well. To this purpose, we consider $s_1 < s_2\le T$.
 By virtue of \cite[Section IV.7]{FlemingS}, $V^\e$ satisfies the dynamic programming principle. Therefore for any $\delta>0$, we can choose a $u_1 \in \cal U$ such that
 \bed  \begin{aligned}V^\e(s_1,x,\al) &  \le \ex \left[\int_{s_1}^{s_2}   \Gamma^\e (t) l(t, X(t),\al(t), u_1(t) ) dt +   V^\e (s_2, X(s_2), \al(s_2))\right] \\ & <  V^\e(s_1,x,\al)  + \delta/3, \end{aligned} \eed
 where $\Gamma^\e (t) = \exp\{-\frac{1}{\e}\int_{s_1}^t \psi^+(X(r),\al(r))dr\}$ and $X=X^{s_1,x,\al;u_1}$.
Then we have from \assumpref{assump-ito-condition}  that
\beq\label{Ve-I-II-III} \begin{aligned} & \abs{V^\e(s_1,x,\al) - V^\e(s_2,x,\al)} -  \delta /3 \\& \ \  \le \abs{\ex\left[\int_{s_1}^{s_2}  \Gamma^\e (t) l(t, X(t),\al(t), u_1(t) ) dt +   V^\e (s_2, X(s_2), \al(s_2))\right]- V^\e(s_2,x,\al)  } \\
 & \ \ \le \int_{s_1}^{s_2} K (1+ \ex \abs{X(t)}^p)dt + \ex \abs{V^\e(s_2, X(s_2), \al) - V^\e(s_2, x, \al) }\\
  & \ \ \ \ \  +
  \ex \abs{V^\e (s_2, X(s_2), \al(s_2)) - V^\e (s_2, X(s_2), \al)}
  \\ & \ \  := I + II + III.
   \end{aligned} \eeq
   Using \eqref{eq-Lp-estimate}, we obtain
   \beq\label{Ve-I} I =\int_{s_1}^{s_2} K (1+ \ex \abs{X(t)}^p)dt \le K(s_2- s_1). \eeq
  For the term  last term, we first notice that   the definition of $V^\e$ in \eqref{eq-V-aux},
  \assumpref{assump-ito-condition}, and \eqref{eq-Lp-estimate} imply that that $V^\e(s_2, X(s_2), j) \le K$ for every $j\in \M$,
  where $K= K(x,T,p)$ is a constant. Then it follows that
\beq\label{Ve-III}\begin{aligned} III  & =\ex \abs{ V^\e (s_2, X(s_2), \al(s_2)) - V^\e (s_2, X(s_2), \al)} \\ & =  \ex \left[\abs{  V^\e (s_2, X(s_2), \al(s_2)) - V^\e (s_2, X(s_2), \al) }I_{\set{\al(s_2)\not= \al}}\right] \\
& \le K \pr\set{\al(s_2)\not= \al} \\
& \le K (s_2-s_1),
\end{aligned}\eeq where in the last inequality, we used \eqref{Q-gen}.
Further, since $$X(s_2) = x + \int_{s_1}^{s_2} b(t, X(t),\al(t),u_1(t)) dt + \int_{s_1}^{s_2} \sigma(t, X(t),\al(t),u_1(t)) dw(t), $$
using \assumpref{assump-ito-condition} and \eqref{eq-Lp-estimate}, we can readily verify that
\bed
\ex\abs{X(s_2)-x} \le K \abs{s_2-s_1}^{1/2}.\eed
Hence by virtue of \eqref{eq-V-x1-x2}, we have
\beq\label{Ve-II} II= \ex \abs{V^\e(s_2,X(s_2),\al) - V^\e(s_2,x,\al)} \le K \ex \abs{X(s_2)- x} \le K\abs{s_2 -s_1}^{1/2}.\eeq Combing the above estimates \eqref{Ve-I}--\eqref{Ve-II} into \eqref{Ve-I-II-III}, we arrive at
\bed \abs{V^\e(s_1,x,\al) - V^\e(s_2,x,\al)} -  \delta /3 \le K \abs{s_2-s_1}^{1/2} + K(s_2-s_1).\eed Since $\delta>0$ is arbitrary, the continuity of $V^\e$ with respect to $s$ is established, as desired.
 \end{proof}

\begin{assumption}\label{assump-continuity-condition}
Suppose there exists a function $\psi:\rr\times \M \mapsto \rr$ satisfying \eqref{psi-Lip} and that
\beq\label{psi<=0-in-[0,infty)}
\psi(y,j) \le 0, \ \ \forall (y,j)\in [0,\infty)\times \M.
\eeq
Moreover, there exists some $u \in U$ such that
\beq\label{continuity-key-condition}  \int_s^t \psi^+(X(r),\al(r)) dr > 0, \text{ a.s. for any } t \in (s,T],\eeq where $X=X^{s,0,\al;u}$ is the controlled process under the constant control $u(t)\equiv u, t \in [s,T]$, $\al=\al^{s,\al}$, $s \in [0,T)$, and  $\al\in \M$.
\end{assumption}

\begin{thm}\label{thm-continuity} In addition to  Assumptions {\em A\ref{assump-ito-condition}} and {\em A\ref{assump-continuity-condition}}, suppose also  that $g(\cdot, \cdot, \al) \in C^{1,2}$ for each $\al \in \M$ and that
\begin{align}\label{continuity-condition-g} & \frac{\partial }{\partial t} g(t,x,\al) + \op^u_t g(t,x,\al) + l(t,x,\al,u) \ge 0, \ \ \forall (t,x,\al, u) \in (0,T) \times (0,\infty) \times \M\times U,
\end{align}
 Then
$V(\cdot,\cdot,\al)$ is continuous for each $\al\in \M$.
\end{thm}

\begin{proof} We divide the proof into two steps. The first step is concerned  with the special case when $l\ge 0$ and $g\equiv 0$ while the second step deals with the general case.

 Step 1.
First assume  $l\ge 0$ and $g\equiv 0$. Fix $(s,\al) \in [0,T) \times \M$.  Let $u  $, $X$, and $\al$ as in \assumpref{assump-continuity-condition}.
 Then \eqref{continuity-key-condition} implies that for any $t> s$
\bed \lim_{\e \downarrow 0} \Gamma^\e(t) = \lim_{\e \downarrow 0} \exp\set{-\frac{1}{\e}\int_s^t \psi^+(X(r),\al(r))dr}=0,\quad \text{a.s.}\eed
and hence by virtue of the dominated convergence theorem and the definition of $J^\e$ in \eqref{eq-cost-aux}, we obtain \bed  \lim_{\e \downarrow 0} J^\e(s,0,\al, u)=0.\eed
Since $J^\e (s,0,\al,u) \ge V^\e(s,0,\al) \ge 0$, it follows that
\bed \lim_{\e \downarrow 0} V^\e(s,0,\al) =0.\eed
Let $0< \e_1 < \e_2$. Then, noting 
  the nonnegativity of the functions $\psi^+$ and  $l$, we can readily verify that
$J^{\e_1}(s,0,\al,u) \le J^{\e_2}(s,0,\al,u)$ and hence $V^{\e_1}(s,0,\al) \le V^{\e_2}(s,0,\al)$.
 We have shown in \thmref{thm-cont-aux} that
  $V^\e$ is continuous. 
Thus Dini's theorem implies that
$\lim_{\e \downarrow 0} V^\e(s,0,\al) =0 $ uniformly on $[0,T]  $  for each $\al \in \M$.
Now let
$$h(\e):= \sup\set{V^\e(s,0,\al): s\in [0,T], \al \in \M}.$$
Then we have
$\lim_{\e \to 0} h(\e)=0$ thanks to the uniform convergence of $V^\e$.

For any $(s,x,\al) \in [0,T] \times (0,\infty)\times \M$, thanks to the dynamic programming principle for $V^\e$
and the definition of $h$, we have\footnote{Note that if $\tau=T$, then $V^\e(T,X(T),\al(T))=V(T,X(T),\al(T))=g(T,X(T),\al(T))=0$ by our assumption on $g$.}
\bea V^\e(s,x,\al) \ad = \inf_{u\in \cal U} \ex \left[\int_s^\tau  l(t, X(t),\al(t),u(t))dt + V^\e(\tau, X(\tau),\al(\tau))\right] \\[2ex]
\ad \le \inf_{u\in \cal U} \ex \int_s^\tau  l(t, X(t),\al(t),u(t))dt + h(\e) \\
\ad = V(s,x,\al) + h(\e).
\eea
Thanks to the definition of $J^\e$ in \eqref{eq-cost-aux} and the assumption that $g\equiv 0$, we have
\bed\begin{aligned} J^\e(s,x,\al,u) & = \ex \int_s^T \Gamma^\e(t) l(t,X(t),\al(t),u)dt\\
   & = \ex\left[\int_s^\tau \Gamma^\e(t) l(t,X(t),\al(t),u(t))dt + I_{\set{\tau< T}} \int^T_\tau \Gamma^\e(t) l(t,X(t),\al(t),u(t))dt \right].\end{aligned} \eed
Note that for all $t\in [s,\tau]$, $X(t)\in [0,\infty)$. Thus it follows from \assumpref{assump-continuity-condition} and the definition of $\Gamma^\e$ in \eqref{eq-gamma-defn} that
$\Gamma^\e(t)=1$ and hence $$J^\e(s,x,\al,u) = J(s,x,\al,u) + \ex \left[I_{\set{\tau < T}}\int^T_\tau \Gamma^\e(t) l(t,X(t),\al(t),u(t))dt\right].$$
Furthermore, since $l \ge 0$, we   have
$$V(s,x,\al) \le V^\e(s,x,\al) \le V(s,x,\al) + h(\e). $$
This implies that $V^\e \to V$ uniformly on $[0,T]\times (0,\infty)\times \M$.
  Since $V^\e$ is continuous, $V$ is also continuous.

  Step 2. For general $l$ and $g$, let $\tilde l(t,x,\al,u)=  l(t,x,\al,u) + \frac{\partial }{\partial t} g(t,x,\al) + \op^u_t g(t,x,\al)  $
  and $\tilde g \equiv 0$.
  Then $\tilde l \ge 0$ by virtue of \eqref{continuity-condition-g} and hence Step 1 implies that the function
  \bed \wdt V(s,x,\al) := \inf_{u\in \cal U}\wdt J(s,x,\al,u) = \inf_{u \in \cal U}\ex \int_s^\tau \tilde l(t,X(t),\al(t), u(t) )dt \eed
  is continuous.  Apply It\^o's formula to $g$,
  \bed \ex g(\tau, X(\tau),\al(\tau)) -g(s,x,\al) = \ex \int_s^\tau \(\frac{\partial }{\partial t} + \op^u_t\) g(t,X(t),\al(t))dt.\eed
  Then it follows that
  \bed \barray \wdt J(s,x,\al,u) \ad = \ex \int_s^\tau \tilde l(t,X(t),\al(t),u(t))dt \\
  \ad = \ex  \int_s^\tau   l(t,X(t),\al(t),u(t))dt + \ex \int_s^\tau  \(\frac{\partial }{\partial t} + \op^u_t\) g(t,X(t),\al(t))dt \\
  \ad = \ex  \int_s^\tau   l(t,X(t),\al(t),u(t))dt + \ex g(\tau, X(\tau),\al(\tau)) -g(s,x,\al) \\
  \ad =J(s,x,\al,u) -g(s,x,\al).
  \earray\eed
  Therefore we conclude that  $V(s,x,\al) = \wdt V(s,x,\al) + g(s,x,\al) $ is also continuous.
\end{proof}

\begin{rem}
With the continuity of the value function at our hands, we have the dynamic programming principle by virtue of \cite{FlemingS}:
\begin{equation}\label{eq:dynamic-programming-principle}\begin{aligned}
V(s,x,\al) =   \inf_{u\cd \in \cal U} \ex \bigg\{\int_s^{\theta\wedge \tau} 
l(t,X(t),\al(t),u(t)) dt    +  
V(\theta\wedge \tau, X(\theta\wedge \tau), \al( \theta\wedge \tau))\bigg\},
\end{aligned}
\end{equation} where $\theta$ is a stopping time.
\end{rem}

 \begin{rem}
 Note that \assumpref{assump-continuity-condition}, in particular \eqref{continuity-key-condition}, is the crucial condition in the proof of \thmref{thm-continuity}.
  One may wonder when \assumpref{assump-continuity-condition} is true? Next we present several sufficient conditions.\end{rem}

\begin{prop}\label{prop-about-A3}
Any one of the following conditions implies Assumption {\em A\ref{assump-continuity-condition}}:
\begin{enumerate}
\item[{\em (i)}] There exists a function $\psi:\rr\times \M \mapsto \rr$ satisfying \eqref{psi-Lip}, \eqref{psi<=0-in-[0,infty)}, and $\psi(0,\al)=0$ for each $\al\in\M$,   and that for some $u\in U$, $\set{\psi(X(t),\al(t)), t\in [s,T]}$
 is a strict local submartingale, where $X=X^{s,0,\al;u}$,  $\al=\al^{s,\al}$, $s\in[0,T)$, and $ \al\in \M$.

 \item[{\em (ii)}] There exists a twice continuously differentiable  function $\psi:\rr\times \M \mapsto \rr$ satisfying \eqref{psi-Lip}, \eqref{psi<=0-in-[0,infty)}, and $\psi(0,\al)=0$ for each $\al\in\M$,   and that
     for some $u\in U$,
  \beq\label{psi-strict-submg} \op_t^u \psi(0,\alpha )= b(t,0,\alpha ,u)\psi'(0,\alpha )
   + \frac{1}{2} \sigma^2(t,0,\alpha ,u) \psi''(0,\alpha )
   >0,  \eeq
  where $t\in[0,T]$ and $\alpha \in \M$.

\item[{\em (iii)}]  There exists a $u\in U$ such that the boundary point $0$ is  {\em regular}  for the domain $(0,\infty)$ for the process $(X,\al)=(X^{s,0,\al;u},\al^{s,\al})$, where $s\in[0,T), \al\in \M$.
\end{enumerate}
\end{prop}

\begin{proof}
(i) Note that $\psi(X(s),\al(s))=\psi(0,\al)=0$.   Since $\set{\psi(X(t),\al(t)), t\in [s,T]}$
 is a strict local submartingale, we have
 $\ex[\psi(X(t),\al(t))|\F_s] > \psi(X(s),\al(s))=0$ a.s. for any $t \in (s,T]$,
 from which \eqref{continuity-key-condition} follows.

 (ii) This is obvious since \eqref{psi-strict-submg} implies that
 \bed \op_t^u \psi(x,i)= b(t,x,i,u)\psi'(x,i) + \frac{1}{2} \sigma^2(t,x,i,u) \psi''(x,i) + \sum_{j=1}^m q_{ij} \psi(x,j) >0,  \ \ t\in[0,T] \eed
 for all $(x,i)\in N\times \M$, where $N$ is a neighborhood of $0$. As a result,
 $\set{\psi(X(t),\al(t)), t\in [s,T]}$
 is a strict local submartingale.

 (iii) If $0$ is a regular boundary point, then the function $\psi(x,\al):=-x$ satisfies the conditions in \assumpref{assump-continuity-condition}. We refer to the appendix and  \thmref{thm-reg-pt} for more discussions on regular boundary point.
\end{proof}

\begin{exm}\label{exm2}
Consider a uncontrolled surplus process given by
\bed dX(t)=2(t-1)dt + (t-X(t))^+ dw(t),\ \ t\in [s,2],\eed
with initial surplus $$X(s)=x>0,$$
where $s\ge 0$  and $w$ is a one-dimensional standard Brownian motion. Similar to \exmref{exm-discontinuous-value-function}, we define $\tau=\tau^{s,x}=\inf\set{t>s: X(t)=0}\wedge 2$ and $V(s,x)= \tau$.

As in \cite{FlemingS}, the signed distance to the boundary point $0$ is $\hat \rho(x)=-x, x\in \rr$. Then
$$ \op_t \hat\rho(0)= -2(t-1)\begin{cases} > 0 & \text{ for } t\in (0,1),\\
                                            <0 & \text{ for } t\in (1,2).\end{cases} $$
Hence the sufficient condition ($\op_t \hat \rho (0) > 0$, for all $t\in[0,2]$)
 for continuity of the value function
given in \cite{FlemingS} fails.  See   \cite[equation (2.8), p. 202]{FlemingS} for more details.

 Nevertheless, we claim that $0$ is a regular boundary point for the domain $(0,\infty)$ and hence \assumpref{assump-continuity-condition} still holds true by virtue of \propref{prop-about-A3}.
 Consequently the value function is continuous.
To this end, we consider the function $\vphi(x)= -  x^2 +   x, x\in  (-\frac{1}{2},\frac{1}{2})$. It is easy to see that $\vphi$ satisfies conditions (i) and (ii) in \thmref{thm-reg-pt}.
Next we show that $\vphi$ satisfies condition (iii) in \thmref{thm-reg-pt} as well.
In fact,
\bed \begin{aligned}\op_t \vphi(0)& = 2(t-1)\vphi'(0) + \frac{1}{2}\((t-0)^+\)^2 \vphi''(0) \\
& = 2(t-1) - t^2 \\
& = -(t-1)^2 -1  \le -1 < 0, \ \ \forall t\in [0,2].\end{aligned}\eed
Then  it follows 
that $\op_t \vphi (x) < 0$ for all $ t\in [0,2]$ and $x\in N $, where $N\subset (-\frac{1}{2},\frac{1}{2})$ is a neighborhood of $0$.
Consequently, $\set{\vphi(X(t)), t\in [s,2]}$ is superharmonic in $(0,\infty)\cap U$.
Therefore \thmref{thm-reg-pt} implies that $0$ is a regular boundary point and hence \assumpref{assump-continuity-condition} is verified.
Further, we can readily verify that all other conditions in \thmref{thm-continuity} are satisfied and hence the value function $V$ is continuous.
\end{exm}

\begin{exm}
Suppose a controlled surplus process $X$ satisfies
\beq\label{exm3-state}  dX(t)= b(t,X(t),\al(t),u(t))dt +\sigma(t,X(t),\al(t),u(t))dw(t), \ \ t \ge s \ge 0, \eeq with initial conditions \bed X(s) =x >0, \ \ \al(s)=\al \in \set{1,2},\eed
where $w$ is a one-dimensional standard Brownian motion, $\al\in \set{1,2}$ is a continuous time Markov chain generated by $Q=\begin{pmatrix} -3 & 3 \\ 4 & -4 \end{pmatrix}$,
$u(t)\in [0,1]$ denotes the retention rate (so $1-u(t)$ is the proportion reinsured to a reinsurance company) at time $t$, and
\bed\begin{aligned}&  b(t,x,1,u)= \sin t +  x + 0.4-0.8 (1-u), &\ \  \sigma(t,x,1,u)= \sin t + 0.5 x+0.5 u, \\ & b(t,x,2,u)= \cos t +  3 x + 1-2 (1-u), &\ \ \sigma(t,x,2,u)= \cos t +  x+2  u .\end{aligned}\eed

Note that \eqref{exm3-state} represents a surplus process subject to
non-cheap reinsurance, investment in a Markovian-modulated Black-Scholes model, and seasonal fluctuations in premium collection.
This is motivated by the model considered in  \cite{Taksar-03}.
Denote $$\tau:= \inf\set{t> s: X(t)=0} \wedge 100.$$ The payoff for a reinsurance strategy $u\cd$ is
\bed J(s,x,\al,u) =\ex_{s,x,\al} \int_s^\tau e^{-rt} X(t)dt,\ \ s\in [0,100), \ x> 0, \ \al=1, 2, \eed where $r>0$ is the discounting factor.
The objective is to maximize the payoff and find a reinsurance strategy $u^*\cd$ such that
\beq\label{exm3-value-fn}
V(s,x,i)= \sup_{u\in \cal U} J(s,x,i,u) = J(s,x,i,u^*), \ \ s\in [0,100), \ x> 0, \ i=1,2.\eeq

We claim that the value function $V$ is continuous with respect to the variables $s$ and $x$ by virtue of \thmref{thm-continuity}. In fact, it is obvious that all conditions in
 Assumptions A\ref{assump-ito-condition}   are satisfied.
Next we use \thmref{thm-reg-pt} and \propref{prop-about-A3} to show that \assumpref{assump-continuity-condition} is also true and hence the claim follows.
To this end, we consider $\vphi(x,1)=-x^2 + 0.5 x$ and  $\vphi(x,2)=-x^2 + 2x$, where $x\in U:=(-0.25,0.25) $. Then we can easily verify that conditions (i) and (ii) in \thmref{thm-reg-pt}
 are satisfied. To verify condition (iii), we let $u=0.5$ and compute
 \bed \barray
 \op^{u=0.5}_t \vphi(0,1)\ad = b(t,0,1,0.5)\vphi'(0,1) + 0.5\sigma^2(t,0,1,0.5)\vphi''(0,1)  -3 \vphi(0,1) + 3 \vphi(0,2)
 \\ \ad =0.5 \sin t + 0.5 (\sin t + 0.5 \cdot 0.5)^2 (-2) \\
 \ad =-\sin^2 t - 0.0625 < 0, \ \text{ for any } t\in[0,100],
 \earray\eed
 and
 \bea \op^{u=0.5}_t \vphi(0,2)\ad = b(t,0,2,0.5)\vphi'(0,2) + 0.5\sigma^2(t,0,2,0.5)\vphi''(0,2)  +4 \vphi(0,1) -4 \vphi(0,2) \\
 \ad = 2\cos t + 0.5 (\cos t + 1)^2 (-2) \\
  \ad = -1 -\cos^2 t < 0, \ \text{ for any } t\in[0,100].
 \eea
 Hence it follows that $\vphi$ is superharmonic   in $\((0,\infty) \cap U\) \times \set{1,2}$ and condition (iii) in \thmref{thm-reg-pt} is verified.
 Thus by virtue of \thmref{thm-reg-pt} and \propref{prop-about-A3}, we conclude that \assumpref{assump-continuity-condition} holds and hence
 $V$ defined in \eqref{exm3-value-fn} is continuous.
\end{exm}


\section{Viscosity Solution}\label{sect-viscosity}
With the continuity of the value function and the dynamic programming principle, we can now characterize the value function to be a
viscosity solution of the HJB equation \eqref{eq:HJB}.
First we recall the notion of viscosity solution from \cite{FlemingS}.

\begin{defn}\label{defn-viscosity-soln}
A function $v$  is called a {\em viscosity subsolution} ({\em viscosity supersolution}, resp.) of \eqref{eq:HJB} if for any $\varphi(\cdot, \cdot, \al) \in C^{1,2}, \al \in \M$, whenever $v-\varphi$ attains a maximum (minimum, resp.) at $(\sxal)$ with $v(\sxal)=\varphi(\sxal)$, we have
\beq\label{eq-vis-sub-defn}
\frac{\partial}{\partial t} \varphi(\sxal) + \inf_{u\in U} \set{\op_s^u \varphi(\sxal) + l(\sxal,u)} \ge 0 \ (\le 0, \text{resp.})
\eeq
Further, a function $v$ is called a {\em viscosity  solution} of \eqref{eq:HJB} if it is both a viscosity subsolution and supersolution of \eqref{eq:HJB}.

\end{defn}

We first state a lemma, whose proof can be found in \cite{Bayraktar}.
\begin{lem}\label{lem-song}
For any $(\sxal)\in [0,T)\times (0,\infty) \times \M$ and $u\in \cal U$, define
\bed \theta:=\inf\set{t>s: X(t)\notin B(x,h)} \wedge (s+h^2), \eed
where $h\in (0,1)$ and $X=X^{\sxal;u}$. Then there exists a positive constant $\kappa$
such that \bed \ex[\theta-s] \ge \kappa h^2.\eed
Moreover, $\kappa=\kappa(s,x) $  is independent of the control $u$.
\end{lem}

\begin{thm}\label{thm-viscosity-soln} Let  Assumptions {\em A\ref{assump-ito-condition}} and {\em A\ref{assump-continuity-condition}} be satisfied.
Then the value function \eqref{eq:value-sw} is a viscosity solution of \eqref{eq:HJB}.
\end{thm}

\begin{proof} The proof is inspired by \cite{Bayraktar}; we use similar ideas. We first establish the viscosity subsolution property of the value function $V$ in Step 1, followed by viscosity supersolution in Step 2.

Step 1. We first prove that $V$ is a viscosity subsolution of \eqref{eq:HJB}. Suppose it was not the case,
then there would exist some $\varphi \in C^{1,2}$, a $u \in U$,
and a maximizer  $(\sxalz)\in [0,T)\times (0,\infty)\times\M$ of $V-\varphi$ with $V(\sxalz)=\varphi(\sxalz)$, but \beq\label{eq-vis-sub-contrad1}
\frac{\partial}{\partial t} \varphi(\sxalz) +  \op_{\sz}^u \varphi(\sxalz) + l(\sxalz,u)< -\delta < 0,
\eeq
where $\delta >0$. Then by the continuity of the function $l(\cdot,\cdot, \al_0,u) + (\frac{\partial}{\partial t}+ \op_t^u)\varphi(\cdot,\cdot, \al_0)$, there exists an $h \in (0,1)$ such that
\beq\label{eq-vis-contrad2} \frac{\partial}{\partial t} \varphi(s,x,\al_0) + \op_s^u \varphi(s,x,\al_0) + l(s,x,\al_0,u) < -\delta/2 < 0,\eeq
 for all $(s,x) \in [s_0,s_0+ h^2) \times B(x_0,h).$
Without loss of generality, we assume that $h<1$ is sufficiently small so that $[s_0,s_0+ h^2) \times B(x_0,h) \subset [0,T) \times (0,\infty)$.
Denote $X=X^{\sxalz;u}$ and define
\bed \theta:=\inf\set{t> s_0: X(t) \notin B(x_0,h)} \wedge (s_0+h^2).\eed
Note that $\theta < \tau$ a.s. By virtue of the dynamic programming principle \eqref{eq:dynamic-programming-principle}, we have
\beq\label{eq-vis-dyna} V(\sxalz) \le \ex \left[\int_{s_0}^\theta l(r,X(r),\al(r),u)dr + V(\theta,X(\theta),\al(\theta))\right].\eeq
Using the assumptions on $\varphi$, we can derive from \eqref{eq-vis-dyna}
that
\bed 0 \le \ex \left[\int_{\sz}
^\theta l(r,X(r),\al(r),u)dr + \varphi(\theta, X(\theta),\al(\theta))-\varphi(\sxalz)\right].\eed
Apply It\^o's formula to $\varphi$,
\bed \ex \varphi(\theta, X(\theta),\al(\theta))-\varphi(\sxalz)
= \ex \int_{\sz}^\theta \(\frac{\partial}{\partial t} + \op^u_r\) \varphi(r,X(r),\al(r))dr.\eed
Hence it follows from \eqref{eq-vis-contrad2} and \lemref{lem-song}  that
\beq\label{eq-vis-sub-eq1} \barray
0\ad \le \ex  \int_{\sz}
^\theta \left[ l(r,X(r),\al(r),u)+\bigg(\frac{\partial}{\partial t} + \op^u_r\bigg) \varphi(r,X(r),\al(r)) \right]dr \\[2ex]
\ad = \ex    \int_{\sz}
^\theta \left[ l(r,X(r),\al_0,u)+\bigg(\frac{\partial}{\partial t} + \op^u_r\bigg) \varphi(r,X(r),\al_0) \right]dr + A \\ [2ex]
\ad \le \ex  \int_{\sz}^\theta (-\frac{\delta}{2})dr + A
  \\[2ex] \ad
\le -\frac{\delta}{2} \kappa h^2 + A,
\earray\eeq
where $\kappa$ is the constant in \lemref{lem-song}, and
\bed \begin{aligned} A= \ex\int_{\sz}^\theta  \bigg[ & l(r,X(r),\al(r),u)+\bigg(\frac{\partial}{\partial t} + \op^u_r\bigg) \varphi(r,X(r),\al(r)) \\  & \ \ - l(r,X(r),\al_0,u)-\bigg(\frac{\partial}{\partial t} + \op^u_r\bigg) \varphi(r,X(r),\al_0) \bigg]dr.
\end{aligned}\eed
Next we show that $A$ is negligible compared to the term $-\frac{\delta}{2}\kappa h^2$. To this end, we denote $$H(\sxal,u):= l(\sxal,u) + \(\frac{\partial}{\partial t}+ \op^u_s\)\varphi(s,x,\al).$$
Then for each $\al \in \M$ and $u \in U$, as a function of $(s,x)$, $H$ is continuous and hence bounded on the compact $[s_0, s_0+1] \times \bar B(x_0,1)$. Therefore we compute from \eqref{Q-gen} that
\beq\label{eq-vis-sub-eq2} \barray
A \ad = \ex \int_{\sz}^\theta [H(r,X(r),\al(r),u) - H(r,X(r),\al_0,u)]dr \\
\ad \le \ex \int_{\sz}^\theta \sum_{j\not= \al_0} \abs{H(r,X(r),j,u)- H(r,X(r),\al_0,u)}I_{\set{\al(r)=j}}dr \\
\ad \le K \int_{\sz}^{\sz+h^2}  \sum_{j\not= \al_0} \pr \set{\al(r)=j|\al(s_0)=\al_0}dr \\
\ad \le K   \int_{\sz}^{\sz+h^2} (r-\sz) dr
 = K h^4,
\earray\eeq
where $K$ is some positive constant independent of $h$ and $u$.
Then it follows from \eqref{eq-vis-sub-eq1} and \eqref{eq-vis-sub-eq2} that for $h>0$ sufficiently small, we have
$$ 0 \le -\frac{\delta}{2} \kappa h^2 + K h^4 < 0.$$
This is a contradiction. Hence the value function $V$ must be a viscosity subsolution of \eqref{eq:HJB}.

Step 2. Now we show that $V$
is a viscosity supersolution of \eqref{eq:HJB}. Again, we use a contradiction argument. Suppose on the contrary that $V$  was not  a viscosity supersolution of \eqref{eq:HJB}. Then there would exist a $\phi \in C^{1,2}$ and a minimizer $(\sxalz)\in [0,T)\times (0,\infty) \times \M$ of $V-\phi$ with $V(\sxalz)=\phi(\sxalz)$, but
\beq\label{vis-super-eq1}\frac{\partial}{\partial t} \phi(\sxalz) + \inf_{u\in U}\set{\op_{s_0}^u \phi(\sxalz)+ l(\sxalz,u)}= \delta >0,  \eeq
where $\delta $ is a constant.
By Assumption A\ref{assump-ito-condition}, 
the function $$(s,x)\mapsto \frac{\partial }{\partial t} \phi(s,x,\al)+ \inf_{u\in U}\set{\op_{s}^u \phi(\sxal)+ l(\sxal,u)}$$ is continuous for each $\al\in \M$. Hence we can find an $h>0 $ such that
\beq\label{vis-super-eq2} \frac{\partial }{\partial t} \phi(s,x,\al_0)+ \inf_{u\in U}\set{\op_{s}^u \phi(s,x,\al_0)+ l(s,x,\al_0,u)} > \frac{\delta}{2}, \ \forall (s,x)\in [s_0,s_0+ h^2) \times B(x_0,h). \eeq
Let $\e= \frac{\delta}{4} \kappa h^2$, where $\kappa$ is the constant in \lemref{lem-song}. Let $u\in \cal U$ be an $\e$-optimal control and denote $X=X^{\sxalz; u}$. Put $\theta:=\inf\set{t> \sz: X(t) \notin B(x_0,h)} \wedge (\sz+h^2)$. Then it follows that
\beq\label{vis-super-ineq0} V(\sxalz) \ge \ex \int_{\sz}^\theta l(r,X(r),\al(r), u(r))dr +\ex V(\theta, X(\theta),\al(\theta)) -\e.\eeq
As in Step 1, using the assumptions and It\^o's formula on $\phi$, we can rewrite \eqref{vis-super-ineq0}
 as \bed \barray
0 \ad \ge \ex \left[ \int_{\sz}^\theta l(r,X(r),\al(r),u(r))dr + \phi(\theta, X(\theta),\al(\theta))- \phi(\sxalz)\right] -\e \\
\ad = \ex  \int_{\sz}^\theta\left[  l(r,X(r),\al(r),u(r)) + \(\frac{\partial}{\partial t} + \op_r^{u(r)}\) \phi(r,X(r),\al(r))\right]dr -\e.
\earray\eed
But $\e= \frac{\delta}{4} \kappa h^2$. This, together with \lemref{lem-song}, leads to
\beq\label{vis-super-eq3} \barray 0 \ad \ge \ex  \int_{\sz}^\theta\left[  l(r,X(r),\al(r),u(r)) + \(\frac{\partial}{\partial t} + \op_r^{u(r)}\) \phi(r,X(r),\al(r))\right]dr - \frac{\delta}{4} \ex [\theta -\sz] \\[2ex]
\ad = \ex  \int_{\sz}^\theta\left[  l(r,X(r),\al(r),u(r)) + \(\frac{\partial}{\partial t} + \op_r^{u(r)}\) \phi(r,X(r),\al(r)) -\frac{\delta}{4}  \right]dr\\[2ex]
\ad = \ex  \int_{\sz}^\theta\left[  l(r,X(r),\al_0,u(r)) + \(\frac{\partial}{\partial t} + \op_r^{u(r)}\) \phi(r,X(r),\al_0) -\frac{\delta}{4}  \right]dr + B,  \earray \eeq
where $$\begin{aligned} B= \ex \int_{\sz}^\theta \bigg[ & l(r,X(r),\al(r),u(r)) + \(\frac{\partial}{\partial t} + \op_r^{u(r)}\) \phi(r,X(r),\al(r)) \\ & -
l(r,X(r),\al_0,u(r)) - \(\frac{\partial}{\partial t} + \op_r^{u(r)}\) \phi(r,X(r),\al_0) \bigg]dr. \end{aligned} $$
Using the same argument as that in Step 1, we deduce that for some constant $K$ independent of $h$ and $u$,  \beq\label{vis-super-eq4}\abs{B} \le K h^4.\eeq
Therefore it follows from \eqref{vis-super-eq2}--\eqref{vis-super-eq4} that
\bed
0    \ge \ex \int_{\sz}^\theta (\frac{\delta}{2}-\frac{\delta}{4}) dr  -  \abs{B}
  \ge \frac{\delta}{ 4}\kappa h^2 - K h^4 >0,
 \eed
 for $h>0$ sufficiently small. This is a contradiction. Therefore $V$ is a viscosity supersolution of \eqref{eq:HJB}. This completes the proof.
\end{proof}


\section{Conclusions and Remarks}\label{sect-conclusions}

In this work, we considered cost optimization  problem  for an insurance company. The surplus of the insurance company was modeled by a controlled regime-switching diffusion.
We presented a sufficient condition for the continuity of the value function and further characterized it as a viscosity solution of the HJB equation
\eqref{eq:HJB}.
Compared with the usual diffusion models, the consideration of regime-switching mechanism   provides a better approximation to the real-world dynamics.
The novelty of this work also includes a new sufficient condition for continuity  of the value function. The
sufficient condition in this paper
is a new generalization of the one in \cite{FlemingS}.

A number of other questions deserve further investigations.
In particular, the next logical step is to
establish a strong comparison result (\cite{CIL92}), which, in turn,
 implies that the value function is the {\em unique} viscosity solution of the HJB equation \eqref{eq:HJB}.
 Consequently, we have the complete characterization of the value function.
  It is conceivable that due to the presence of regime switching, the analysis will be much more involved than
  that in the user's guide \cite{CIL92}.
  Also,
we were not able  to obtain the explicit form of the value function and an optimal control by solving
\eqref{eq:HJB}.
The reason for this is that \eqref{eq:HJB} is a coupled system of nonlinear
second order partial differential equations, rendering
extreme difficulty in finding a closed form solution
of \eqref{eq:HJB}. 
Therefore a viable alternative is to employ numerical approximations.
The Markov chain approximation method developed in \cite{Kushner-D} will be
utilized in the near future.
We may also consider  more complicated stochastic models such as
regime-switching diffusions with jumps as well as other optimality criteria
such as dividend maximization and ruin probability minimization problems.

 \appendix\section{Regular Boundary Point}\label{sect-appendix}
 This appendix provides a result on regular boundary points. For notational simplicity, we shall present  the  result when the continuous component $X$
 is 1-dimensional. The multi-dimensional case can be handled in a similar fashion.
 Note that the notations in the appendix are not necessarily the same as  those in the main part of the paper.

 Let $(X,\al)\in \rr \times \M$ be a switching diffusion process, where $ \M=\set{1,\dots, m}$.
The generator $\cal G$ of $(X,\al)$ is defined as follows. For any $h(\cdot, \cdot, i)\in C^{1,2}, i\in \M$, we define
\beq\label{app-generator}\begin{aligned} {\cal G} h(t,x,i) =&  \frac{\partial }{\partial t} h(t,x,i)+   h'(t,x,i) b(t,x,i)  \\ & + \frac{1}{2}  h''(t,x,i) \sigma^2(t,x,i)    + \sum_{j=1}^m q_{ij} h(t,x,j), \end{aligned}\eeq
 where $  h'$ and $ h''$ denote the first and second order derivatives  of $h$ with respect to the variable $x$, respectively, $b,\sigma:[0,\infty)\times\rr\times \M \mapsto \rr$ are given functions,
and $q_{ij}$ are constants satisfying $q_{ij} \ge 0$ for $i\not =j$ and $q_{ii}=-\sum_{j\not= i}q_{ij}$.
Further, we assume that   $b $ and $\sigma  $ satisfy
\beq\label{app-ito}
\begin{aligned}
&\abs{b(t,x,i)-b(t,y,i)} + \abs{\sigma(t,x,i)-\sigma(t,y,i)} \le K \abs{x-y}, \\ & \abs{b(t,x,i)} + \abs{\sigma(t,x,i)} \le K (1+\abs{x}), \ \
 \forall t\in [0, \infty),
 x,y\in \rr, i \in \M,
\end{aligned}
\eeq
where $K$ is a positive constant. It is well-known (\cite{MaoY}) that under these conditions, for any $s\ge 0$, $x\in \rr $, and $\al \in \M$, the generator \eqref{app-generator} uniquely determines a switching process $(X\cd,\al\cd)$ with initial conditions $X(s)=x$ and $\al(s)=\al$. Denote such a process by $(X^{s,x,\al}, \al^{s,\al})$ if the emphasis on initial conditions are needed.

 Let $G $ be an open subset of $\rr $ and $a\in \partial G$. The point $a$
 is said to be  {\em regular}  for the process $(X,\al)$ in $G$ if for any
 $s\ge 0$ and $\al\in \M$ we have
 \bed \pr \set{\tau=s}=1,\eed
 where $$\tau=\tau^{s,a,\al}:=\inf\set{t>s: X^{s,a,\al}(t) \notin G}$$ denotes the first exit time for the process $(X,\al)$ from $G$.

 \begin{thm}\label{thm-reg-pt}
The point $a\in \partial G$ is a regular point if
there exist a neighborhood $U$ of $a$ and  a function $\varphi: U\times \M \mapsto \rr$ such that
\begin{itemize}
\item[{\em (i)}] $\varphi(x,i) > 0$ for all $x\in G\cap U-\set{a}$ and each $i \in \M$;
\item[{\em (ii)}] $\lim_{x\to a, x\in G} \varphi(x,i)=0$ for each $i\in \M$; and
\item[{\em (iii)}] $\varphi$ is {\em superharmonic} in $(G\cap U) \times \M$, that is, $\varphi$ is bounded below and continuous in $(G\cap U) \times \M$ and satisfies
\beq\label{app-su-homonic} \varphi(x,\alpha ) \ge \ex  \varphi(X^{s,x,\alpha }(\tau_V),\al^{s,\alpha }(\tau_V)), \ \ \forall (x,\alpha )\in V, \eeq
where $s\ge 0$, $V \subset (G\cap U) \times M$, and $\tau_V=\inf\set{t>s: (X^{s,x,\alpha }(t),\al^{s,\alpha }(t))\notin V}$.
\end{itemize}
\end{thm}

\begin{proof} The proof is motivated by \cite[Chapter 13]{Dynkin-II}, we use similar ideas.
We divide the proof into  several steps.

Step 1. Let $U$ and $\varphi$ be as in the statement of the theorem. Without loss of generality, we assume \beq\label{app-vphi-1}\sup_{(x,i)\in U \times \M}\vphi(x,i) \le 1 .\eeq
In fact, it is not hard to see that if $\vphi$ satisfies (i)--(iii), then so does the function $\vphi\wedge 1 =\min(\vphi,1)$.

Suppose that $a$ is not regular, then we have
\beq\label{app-contradiction}
\pr  \set{\tau >\sz} >0 \text{ for some }\sz\ge 0 \text{ and }\ell \in \M,
\eeq where $\tau=\tau^{\sz,a,\ell}=\inf\set{t>\sz: X^{\sz,a,\ell}(t) \notin G}$.
Then, by virtue of the Blunmenthal Zero-One Law (\cite{Karatzas-S}),
\beq\label{app-0-1}
\pr  \set{\tau >\sz} =1.
\eeq
Obviously, \bed \set{X(\sz)=a, \al(\sz)=\ell,\tau> \sz} \subset \set{\sup_{t\in[\sz,\tau)}\dist(a,X(t)) >0}.\eed
Hence it follows that
\bed \pr  \set{\sup_{t\in[\sz,\tau)}\dist(a,X(t)) >0} >0.\eed
Therefore for some $\delta >0$, we have
\beq\label{app-dist>delta} \pr  \set{\sup_{t\in[\sz,\tau)}\dist(a,X(t)) >\delta} >0.\eeq

Step 2.
Now set $G_0 =G\cap B(a,\delta)$ and $\tau_0=\inf\set{t>\sz, X(t) \notin G_0}$. Note that by choosing $\delta$ sufficiently small, we may without loss of generality assume  that $G_0 \subset U$.
Now  for any $t> \sz$, we can write
\beq\label{app-I1+I2} \begin{aligned}
& \ex \left[I_{\set{\tau_0 < \tau}} \vphi(X(\tau_0),\al(\tau_0))\right] \\
 & \ \ = \ex\left[I_{\set{\tau_0 < \tau,\tau_0 \le t}} \vphi(X(\tau_0),\al(\tau_0))\right] + \ex\left[I_{\set{\tau_0 < \tau, \tau_0>t}} \vphi(X(\tau_0),\al(\tau_0))\right] \\
  &\ \, := I_1(t)+ I_2(t).
\end{aligned}\eeq
Using \eqref{app-vphi-1}, \eqref{app-0-1}, and the continuity of the sample paths of $X$, we have
\beq\label{app-I1-to-0}
\lim_{t\downarrow \sz} I_1(t) \le \lim_{t \downarrow \sz} \pr\set{ \tau_0 \le t} \le \pr\set{\tau_0 =\sz} =0.
\eeq
On the other hand, it follows from the strong Markov property and \eqref{app-su-homonic} that
\beq\label{app-I2-eq1} \begin{aligned}
I_2(t) &= \ex\left[I_{\set{\tau_0 < \tau, \tau_0 >t}} \vphi(X(\tau_0),\al(\tau_0))\right]\\
 & = \ex\left[I_{\set{\tau_0 > t}} \ex_{X(t),\al(t)}[I_{\set{\tau_0 < \tau}} \vphi(X(\tau_0),\al(\tau_0))]\right] \\
 & \le \ex\left[I_{\set{\tau_0 > t}} \vphi(X(t),\al(t))\right].
\end{aligned}\eeq
Thanks to condition (ii), for any $\e >0$, we can choose a neighborhood $N$ of $a$ such that
\beq\label{app-I2-eq2} \vphi(x,i) < \e, \text{ for any }(x,i)\in (N\cap G) \times \M. \eeq
Also, since the sample paths of $X$ are continuous (see, for example, \cite{MaoY} or \cite{YZ-10}), we can choose some $D \subset N$ such that
\beq\label{app-I2-eq3}\pr \set{\beta > \sz} =1, \text{ where } \beta=\inf\set{t> \sz: X(t) \notin D}.\eeq
Then it follows from \eqref{app-vphi-1}, \eqref{app-I2-eq1}, and \eqref{app-I2-eq2} that
\bed \begin{aligned}
I_2(t)
 & \le \ex\left[I_{\set{ \tau_0 > t}}\vphi(X(t),\al(t))\right]\\
 & = \ex\left[I_{\set{\tau_0 > t, \beta > t}}\vphi(X(t),\al(t))\right] + \ex\left[I_{\set{\tau_0 > t, \beta \le t}}\vphi(X(t),\al(t))\right] \\
 & \le \e + \pr\set{\beta\le t}.
\end{aligned}\eed
By virtue of \eqref{app-I2-eq3}, we further obtain
$\limsup_{t\downarrow \sz} I_2(t) \le \e$. But $\e>0$ is arbitrary, it therefore follows that
\beq\label{app-I2-0} \lim_{t\downarrow \sz} I_2(t) =0.\eeq
A combination of \eqref{app-I1+I2}, \eqref{app-I1-to-0}, and \eqref{app-I2-0} leads to
\beq\label{app-contradiction-eq1}
 \ex \left[I_{\set{\tau_0 < \tau}} \vphi(X(\tau_0),\al(\tau_0))\right]=0.
\eeq

Step 3. Now we set $A= \set{\tau_0 < \tau, X(\tau_0) \in G}$. Then it is obvious that
$$\set{\sup_{t\in[\sz,\tau)}\dist(a,X(t)) > \delta} \subset A.$$
This, together with \eqref{app-dist>delta}, implies that $\pr(A) > 0$. Therefore it follows from condition (i) that
\beq\label{app-contradiction-eq2} \ex\left[I_{\set{\tau_0 < \tau}}\vphi(X(\tau_0),\al(\tau_0))\right] \ge \ex\left[\vphi(X(\tau_0),\al(\tau_0)) I_A\right] > 0.\eeq
Finally, the contradiction between \eqref{app-contradiction-eq1} and \eqref{app-contradiction-eq2} implies that $a$ must be a regular boundary point.
\end{proof}

\begin{rem}
If the process $(X,\al)$ is assumed to be strong Feller (\cite{ZY-09b}), then we can show that the conditions in \thmref{thm-reg-pt} are also necessary. The argument is similar to that in \cite[Chapter 13]{Dynkin-II}. We shall omit the details here.
\end{rem}

\bibliography{refs}

\def\cprime{$'$}
\begin{thebibliography}{}

\bibitem[{Ait Rami} et~al., 2001]{Rami-01}
{Ait Rami}, M., Moore, J., and Zhou, X. (2001).
\newblock Indefinite stochastic linear quadratic control and generalized
  differential {R}iccati equation.
\newblock {\em SIAM J. Control Optim.}, 40(4):1296--–1311.

\bibitem[Asmussen, 1989]{Asmussen-89}
Asmussen, S. (1989).
\newblock Risk theory in a {M}arkovian environment.
\newblock {\em Scand. Acturial Journal}, pages 69--100.

\bibitem[Bayraktar et~al., 2010]{Bayraktar}
Bayraktar, E., Song, Q., and Yang, J. (2010).
\newblock On the continuity of stochastic exit time control problems.
\newblock {\em Stochastic Analysis and Applications}.
\newblock to appear, \textcolor{blue}{\tt arXiv:0907.0062.}

\bibitem[Cadenillas et~al., 2006]{Cadenillas-06}
Cadenillas, A., Choulli, T., Taksar, M., and Zhang, L. (2006).
\newblock Classical and impulse stochastic control for the optimization of the
  dividend and risk policies of an insurance firm.
\newblock {\em Math. Finance}, 16(1):181--–202.

\bibitem[Cai et~al., 2009]{Cai-Feng}
Cai, J., Feng, R., and Willmot, G. (2009).
\newblock On the expectation of total discounted operating costs up to default
  and its applications.
\newblock {\em Adv. in Appl. Probab.}, 41(2):495--–522.

\bibitem[Choulli et~al., 2003]{Choulli}
Choulli, T., Taksar, M., and Zhou, X. (2003).
\newblock A diffusion model for optimal dividend distribution for a company
  with constraints on risk control.
\newblock {\em SIAM J. Control Optim.}, 41(6):1946–--1979.

\bibitem[Crandall et~al., 1992]{CIL92}
Crandall, M.~G., Ishii, H., and Lions, P.-L. (1992).
\newblock User's guide to viscosity solutions of second order partial
  differential equations.
\newblock {\em Bull. Amer. Math. Soc. (N.S.)}, 27(1):1--67.

\bibitem[Crandall and Lions, 1983]{Crandall-Lions-83}
Crandall, M.~G. and Lions, P.-L. (1983).
\newblock Viscosity solutions of {H}amilton-{J}acobi equations.
\newblock {\em Trans. Amer. Math. Soc.}, 277:1--42.

\bibitem[Dynkin, 1965]{Dynkin-II}
Dynkin, E. (1965).
\newblock {\em Markov Processes}, volume~II.
\newblock Springer-Verlag, Berlin.

\bibitem[Fleming and Soner, 2006]{FlemingS}
Fleming, W. and Soner, H. (2006).
\newblock {\em Controlled Markov Processes and Viscosity Solutions}.
\newblock Springer-Verlag, New York, NY, second edition.

\bibitem[Grandell, 1991]{Grandell-91}
Grandell, J. (1991).
\newblock {\em Aspects of Risk Theory}.
\newblock Springer-Verlag, New York.

\bibitem[Hespanha, 2005]{Hespanha05}
Hespanha, J. (2005).
\newblock A model for stochastic hybrid systems with application to
  communication networks.
\newblock {\em Nonlinear Anal.}, 62:1353--1383.

\bibitem[Irgens and Paulsen, 2005]{Irgens-05}
Irgens, C. and Paulsen, J. (2005).
\newblock Maximizing terminal utility by controlling risk exposure: a
  discrete-time dynamic control approach.
\newblock {\em Scand. Actuar. J.}, (2):142--–160.

\bibitem[Karatzas and Shreve, 1991]{Karatzas-S}
Karatzas, I. and Shreve, S. (1991).
\newblock {\em Brownian Motion and Stochastic Calculus}.
\newblock Springer, New York, second edition.

\bibitem[Kushner and Dupuis, 2001]{Kushner-D}
Kushner, H. and Dupuis, P. (2001).
\newblock {\em Numerical Methods for Stochstic Control Problems in Continuous
  Time}.
\newblock Springer, New York, second edition.

\bibitem[Li et~al., 2003]{Li-03}
Li, X., Zhou, X., and {Ait Rami}, M. (2003).
\newblock Indefinite stochastic linear quadratic control with {M}arkovian jumps
  in infinite time horizon.
\newblock {\em Journal of Global Optimization}, 27:149--175.

\bibitem[Lions, 1983]{Lio83b}
Lions, P.-L. (1983).
\newblock Optimal control of diffusion processes and
  {H}amilton-{J}acobi-{B}ellman equations. {I}. {T}he dynamic programming
  principle and applications.
\newblock {\em Comm. Partial Differential Equations}, 8(10):1101--1174.

\bibitem[Mao and Yuan, 2006]{MaoY}
Mao, X. and Yuan, C. (2006).
\newblock {\em Stochastic Differential Equations with {M}arkovian Switching}.
\newblock Impreial College Press, London.

\bibitem[Mariton, 1990]{Mariton}
Mariton, M. (1990).
\newblock {\em Jump Linear Systems in Automatic Control}.
\newblock Marcel Dekker, Inc., New York.

\bibitem[Paulsen and Gjessing, 1997]{Paulsen-97}
Paulsen, J. and Gjessing, H. (1997).
\newblock Optimal choice of dividend barriers for a risk process with
  stochastic return on investments.
\newblock {\em Insurance Math. Econom.}, 20(3):215--–223.

\bibitem[Paulsen et~al., 2005]{Paulsen-05}
Paulsen, J., Kasozi, J., and Steigen, A. (2005).
\newblock A numerical method to find the probability of ultimate ruin in the
  classical risk model with stochastic return on investments.
\newblock {\em Insurance: Mathematics and Economics}, 36:399--420.

\bibitem[Schmidli, 2001]{Schmidli-01}
Schmidli, H. (2001).
\newblock Optimal proportional reinsurance policies in a dynamic setting.
\newblock {\em Scand. Actuar. J.}, (1):55--68.

\bibitem[Schmidli, 2002]{Schmidli-02}
Schmidli, H. (2002).
\newblock On minimizing the ruin probability by investment and reinsurance.
\newblock {\em Ann. Appl. Probab.}, 12(3):890--–907.

\bibitem[Song et~al., 2010]{Song-S-Z}
Song, Q., Stockbridge, R., and Zhu, C. (2010).
\newblock {\em On Optimal Harvesting Problems in Random Environments}.
\newblock preprint,
  \href{http://arxiv.org/abs/1004.4250v2}{\textcolor{blue}{\tt
  arXiv:1004.4250v2}}.

\bibitem[Taksar and Hunderup, 2007]{Taksar-07}
Taksar, M. and Hunderup, C. (2007).
\newblock The influence of bankruptcy value on optimal risk control for
  diffusion models with proportional reinsurance.
\newblock {\em Insurance Math. Econom.}, 40(2):311--321.

\bibitem[Taksar and Markussen, 2003]{Taksar-03}
Taksar, M. and Markussen, C. (2003).
\newblock Optimal dynamic reinsurance policies for large insurance portfolios.
\newblock {\em Finance Stoch.}, 7(1):97--121.

\bibitem[Taksar and Zeng, 2009]{Taksar-Z}
Taksar, M. and Zeng, X. (2009).
\newblock On maximizing {CRRA} utility in regime switching markets with random
  endowment.
\newblock {\em SIAM J. Control Optim.}, 48(5):2984--3002.

\bibitem[Touzi, 2000]{Touzi-00}
Touzi, N. (2000).
\newblock Optimal insurance demand under marked point processes shocks.
\newblock {\em Ann. Appl. Probab.}, 10(1):283--312.

\bibitem[Yang and Yin, 2004]{Yang-Yin04}
Yang, H. and Yin, G. (2004).
\newblock Ruin probability for a model under markovian switching regmie.
\newblock In Lai, T., Yang, H., and Yung, S., editors, {\em Probability,
  Finance and Insurance}, pages 206--217. World Scientific, River Edge, NJ.

\bibitem[Yin et~al., 2002]{YinLZ}
Yin, G., Liu, R., and Zhang, Q. (2002).
\newblock Recursive algorithms for stock liquidation: A stochastic optimization
  approach.
\newblock {\em SIAM J. Optim.}, 13:240--263.

\bibitem[Yin and Zhu, 2010]{YZ-10}
Yin, G. and Zhu, C. (2010).
\newblock {\em Hybrid Switching Diffusions: Properties and Applications}.
\newblock Springer, New York.

\bibitem[Yong and Zhou, 1999]{YongZ}
Yong, J. and Zhou, X. (1999).
\newblock {\em Stochastic Controls: Hamiltonian Systems and HJB Equations}.
\newblock Springer-Verlag, New York.

\bibitem[Zhou and Yin, 2003]{Zhou-Yin}
Zhou, X. and Yin, G. (2003).
\newblock Markowitz's mean variance portfolio selection with regime switching:
  {A} continuous-time model.
\newblock {\em SIAM J. Control Optim.}, 42:1466--–1482.

\bibitem[Zhu and Yin, 2009]{ZY-09b}
Zhu, C. and Yin, G. (2009).
\newblock On strong feller, recurrence, and weak stabilization of
  regime-switching diffusions.
\newblock {\em SIAM J. Control Optim.}, 48:2003--2031.

\end{thebibliography}

\end{document}